\documentclass[12pt]{amsart}
\usepackage[colorlinks=true, pdfstartview=FitV, linkcolor=blue, citecolor=blue, urlcolor=blue]{hyperref}

\usepackage{amssymb,amsmath, amscd}
\usepackage{times, verbatim}
\usepackage{graphicx}
\usepackage[english]{babel}
 \usepackage[usenames, dvipsnames]{color}
\usepackage{amsmath,amssymb,amsfonts}
\usepackage{enumerate}
\usepackage{anysize}
\marginsize{3cm}{3cm}{3cm}{3cm}
\input xy
\xyoption{all}
\usepackage{pb-diagram}
\usepackage[all]{xy}
\input xy
\xyoption{all}

\DeclareFontFamily{OT1}{rsfs}{}
\DeclareFontShape{OT1}{rsfs}{n}{it}{<-> rsfs10}{}
\DeclareMathAlphabet{\mathscr}{OT1}{rsfs}{n}{it}

\newcommand{\la}{\mathfrak{p}}
\DeclareMathOperator{\sgn}{sgn}

\makeatletter
\let\@wraptoccontribs\wraptoccontribs
\makeatother

\begin{document}
\theoremstyle{plain}

\newtheorem{theorem}{Theorem}[section]
\newtheorem{thm}[equation]{Theorem}
\newtheorem{prop}[equation]{Proposition}
\newtheorem{cor}[equation]{Corollary}
\newtheorem{conj}[equation]{Conjecture}
\newtheorem{lemma}[equation]{Lemma}
\newtheorem{definition}[equation]{Definition}
\newtheorem{question}[equation]{Question}
\theoremstyle{definition}
\newtheorem{remark}[equation]{Remark}
\newtheorem{example}[equation]{Example}
\numberwithin{equation}{section}

\newcommand{\Hecke}{\mathcal{H}}
\newcommand{\Liea}{\mathfrak{a}}
\newcommand{\Cmg}{C_{\mathrm{mg}}}
\newcommand{\Cinftyumg}{C^{\infty}_{\mathrm{umg}}}
\newcommand{\Cfd}{C_{\mathrm{fd}}}
\newcommand{\Cinftyfd}{C^{\infty}_{\mathrm{ufd}}}
\newcommand{\sspace}{\Gamma \backslash G}
\newcommand{\Sym}{\mathbb{S}}

\newcommand{\PP}{\mathcal{P}}
\newcommand{\bfP}{\mathbf{P}}
\newcommand{\bfQ}{\mathbf{Q}}
\newcommand{\Siegel}{\mathfrak{S}}
\newcommand{\g}{\mathfrak{g}}
\newcommand{\A}{{\rm A}}
\newcommand{\B}{{\rm B}}
\newcommand{\Q}{\mathbb{Q}}
\newcommand{\Gm}{\mathbb{G}_m}
\newcommand{\kk}{\mathfrak{k}}
\newcommand{\nn}{\mathfrak{n}}
\newcommand{\tF}{\tilde{F}}
\newcommand{\p}{\mathfrak{p}}
\newcommand{\m}{\mathfrak{m}}
\newcommand{\bb}{\mathfrak{b}}
\newcommand{\Ad}{{\rm Ad}\,}
\newcommand{\ttt}{\mathfrak{t}}
\newcommand{\frakt}{\mathfrak{t}}
\newcommand{\U}{\mathcal{U}}
\newcommand{\Z}{\mathbb{Z}}
\newcommand{\bfG}{\mathbf{G}}
\newcommand{\bfT}{\mathbf{T}}
\newcommand{\R}{\mathbb{R}}
\newcommand{\ST}{\mathbb{S}}
\newcommand{\h}{\mathfrak{h}}
\newcommand{\bC}{\mathbb{C}}
\newcommand{\C}{\mathbb{C}}
\newcommand{\E}{\mathbb{E}}
\newcommand{\F}{\mathbb{F}}
\newcommand{\N}{\mathbb{N}}
\newcommand{\qH}{\mathbb {H}}
\newcommand{\temp}{{\rm temp}}
\newcommand{\Hom}{{\rm Hom}}
\newcommand{\Aut}{{\rm Aut}}
\newcommand{\Ext}{{\rm Ext}}
\newcommand{\Nm}{{\rm Nm}}
\newcommand{\End}{{\rm End}\,}
\newcommand{\Ind}{{\rm Ind}\,}
\def\circG{{\,^\circ G}}
\def\M{{\rm M}}
\def\diag{{\rm diag}}
\def\Ad{{\rm Ad}}
\def\G{{\rm G}}
\def\H{{\rm H}}
\def\SL{{\rm SL}}
\def\PSL{{\rm PSL}}
\def\GSp{{\rm GSp}}
\def\PGSp{{\rm PGSp}}
\def\Sp{{\rm Sp}}
\def\St{{\rm St}}
\def\GU{{\rm GU}}
\def\SU{{\rm SU}}
\def\U{{\rm U}}
\def\GO{{\rm GO}}
\def\GL{{\rm GL}}
\def\PGL{{\rm PGL}}
\def\GSO{{\rm GSO}}
\def\Gal{{\rm Gal}}
\def\SO{{\rm SO}}
\def\O{{\rm  O}}
\def\sym{{\rm sym}}
\def\St{{\rm St}}
\def\tr{{\rm tr\,}}
\def\ad{{\rm ad\, }}
\def\Ad{{\rm Ad\, }}
\def\rank{{\rm rank\,}}

\subjclass{Primary 11F70; Secondary 22E55}

\title[A character relationship between symmetric group and hyperoctahedral group]
{      A character relationship between symmetric group and hyperoctahedral group  }

\author{Frank L\"ubeck and Dipendra Prasad}

\contrib[with an appendix by]{Arvind Ayyer}

\address{Lehrstuhl D f\"ur Mathematik, Pontdriesch 14/16, 52062 Aachen, Germany}
\email{frank.luebeck@math.rwth-aachen.de}

\address{Indian Institute of Technology Bombay, Powai, Mumbai-400076} 
\address{Tata Institute of Fundamental
Research, Colaba, Mumbai-400005.}
\email{prasad.dipendra@gmail.com}

\address{Indian Institute of Science, Bengaluru-560012}
\email{arvind@iisc.ac.in}
\begin{abstract} {We relate character theory  of  the symmetric groups $\Sym_{2n}$ and $\Sym_{2n+1}$ with that of
  the hyperoctahedral group $B_n = (\Z/2)^n \rtimes \Sym_n$, as part of the expectation that the character theory of
  reductive groups with diagram automorphism and their Weyl groups, is related to the character theory of the fixed
  subgroup of the diagram automorphism.} \end{abstract}

\maketitle
    {\hfill \today}
    
\tableofcontents

\section{Introduction} Let $G$ be either the symmetric group $\Sym_{2n}$ or the symmetric group $\Sym_{2n+1}$, and $H$ the hyperoctahedral group $B_n = (\Z/2)^n \rtimes \Sym_n$,
sitting naturally inside $G$ ($B_n \subset \Sym_{2n} \subset \Sym_{2n+1}$)
as the centralizer of a fixed point free involution $w_0$ in $\Sym_{2n}$. In this paper, we take
the symmetric group $\Sym_{2n}$ as acting on the set $\{\pm 1,\pm 2,\cdots, \pm n \}$ of cardinality $2n$,
and the symmetric group $\Sym_{2n+1}$ as acting on the set $\{0, \pm 1,\pm 2,\cdots, \pm n \}$ of cardinality $2n+1$.
We fix $w_0= (1,-1)(2,-2)\cdots  (n, -n)$.
The paper proves a character
relationship between the irreducible representations of the groups $G$ and $B_n$, closely related to character identities available between (finite dimensional, irreducible,
algebraic) representations of the groups $\GL_{2n}(\C)$, or $\GL_{2n+1}(\C)$ which are self-dual, i.e., invariant under the involution
$g \rightarrow {}^tg^{-1}$, with (finite dimensional, irreducible,
algebraic) representations of the groups $\SO_{2n+1}(\C)$, or $\Sp_{2n}(\C)$, 
for which we refer to \cite{KLP}. These character identities are classically known
as Shintani character identities, first observed between representations of $\GL_n(\F_q)$ and $\GL_n(\F_{q^d})$, cf. \cite{Sh}, although for the case at hand, it
would be much closer to consider irreducible unipotent representations of  say $\U_{2n}(\F_q)$ corresponding to an irreducible
representation of the Weyl group $B_n$ and the associated basechanged representation of $\GL_{2n}(\F_{q^2})$ associated to a representation
of the Weyl group $\Sym_{2n}$, and which are
related by a basechange character identity, cf. \cite{Ka}, \cite{D}.

Observe that in all the cases above, the group $G$ comes equipped with an automorphism, call it
$j$, of finite order
(such as conjugation by $w_0$
for symmetric groups), and $H$ is either the subgroup of fixed points of this automorphism, or is closely related to the subgroup of fixed points
(through a dual group construction as in \cite{KLP}), and in the theory of basechange,
it is not so much the representation theory of the group $G$ which is important, but rather it is the representation theory of the group $G \rtimes \langle j \rangle$. 
Basic to character identities in all the cases referred above, are two maps (correspondence in general):

\begin{enumerate}
\item A correspondence $BC$ between irreducible representations of $H$ and  irreducible representations of $G$,
  to be called basechange of representations
  of $H$ to representations of $G$, which in all the cases listed above is an injective map from
  irreducible representations of $H$ to  irreducible representations of $G$ (but not surjective except in some trivial cases).

\item  A correspondence $\Nm$ between  (certain) $j$-twisted-conjugacy classes in $G$ (two elements $g_1,g_2 $ in $G$ are said to be $j$-twisted-conjugate if there exists $g\in G$ such that
  $g_1 = gg_2jg^{-1}j^{-1}$) and  (certain) conjugacy classes in $H$, to be called the norm mapping from $G$ to $H$, and denoted $\Nm: G\rightarrow H$.

  \end{enumerate}

In the various examples mentioned above, it is the map
$BC$ between irreducible representations of $H$ and  irreducible representations of $G$ that is non-trivial to define or construct, whereas the norm map
$\Nm: G\rightarrow H$ is usually straightforward to define, and is defined on {\it all} of $G$. However, in the case at hand of symmetric groups,
norm mapping is also not so trivial to define especially as it is not defined on all of $G$.

The main theorem of this work, Theorem \ref{main}, proves the following character relation:
$$\Theta(BC(\pi))(w) = \epsilon(BC (\pi))
    \Theta(\pi)(\Nm w),$$
    for any irreducible representation $\pi$ of $B_n$, and $w \in  G$ which is a product of even cycles with at most one fixed
    point    (where    $ \epsilon(BC (\pi)) = \pm 1$).
        In particular,
    taking $w=w_0$, we have 
    $$\Theta(BC(\pi))(w_0) = \epsilon(BC (\pi))   \Theta(\pi)(1),$$
    which was at the origin of this work. After the completion of this work, Prof. G. Lusztig informed the authors that this
    special case occurs on page 110 of his paper \cite{Lu}.

The above theorem was arrived at  by computations done via the GAP software~\cite{GAP}, and inspired by the hope that
basechange character
identities available in many situations involving reductive groups, also have an analogue for Weyl groups of these algebraic groups. We eventually found that it is a simple consequence of Theorem \ref{little} due to Littlewood in \cite {Li1}
for even symmetric groups. Along the way, we have given a complete proof of this theorem of Littlewood both for $\Sym_{2n}$ as well as $\Sym_{2n+1}$. It is surprising 
that the desired character identity relating symmetric groups and     hyperoctahedral group is proved via Frobenius character formula for symmetric groups which is a form of Schur-Weyl duality by a ``factorization'' of the character formula (i.e., Schur polynomials)
for irreducible representations of $\GL_{2n}(\C)$ on special elements discovered by the second author in \cite{P}, which is already
there in \cite{Li1} written in 1940!

    \section{The Norm mapping}

    The aim of this section is to define a map from certain conjugacy classes in $G$ to conjugacy classes in $H$
 where $G$ is either the symmetric group $\Sym_{2n}$ or the symmetric group $\Sym_{2n+1}$, and $H$ is the  hyperoctahedral group $B_n = (\Z/2)^n \rtimes \Sym_n$,
 sitting naturally inside $G$ as the centralizer of the fixed point free involution $w_0= (1,-1)(2,-2)\cdots  (n, -n)$. We will call
 this map the norm mapping and denote it as $\Nm: G\rightarrow H$ even though it is not defined on all of $G$.
 
 Dealing with both the groups $\Sym_{2n}$ and $\Sym_{2n+1}$ at the same time has some notational issues, so  
 we first deal with $G=\Sym_{2n}$, and then come back to $G=\Sym_{2n+1}$ at the end. 
We let  $B_n = (\Z/2)^n \rtimes \Sym_n$ be defined more concretely as:
$$B_n = (\Z/2)^n \rtimes \Sym_n = \{ h \in \Aut \{\pm 1, \pm 2 \cdots, \pm n\} | h(-i)=-h(i), \hspace{.5cm}  \forall i \},$$
with $\Sym_n \subset B_n$ be the subgroup of permutations $\sigma$ of $  \{ 1,  2 \cdots,  n\}$ extended to the points $\{-1,\ldots, -n\}$ by
$\sigma(-i)= -\sigma(i)$.

 We begin by recalling that conjugacy classes in the symmetric group $\Sym_m$ are in bijective correspondence with partitions $\p$  of $m$, which are collection
 of nonzero integers $\{p_1 \geq p_2\geq \cdots \geq p_r \geq 1\}$.

 Conjugacy classes in $B_n$
 are also well-understood beginning with the observation that there are two $B_n$-conjugacy classes among elements of $B_n$ which
 project to an $n$-cycle, say $s_n=(123\cdots n)$, in $\Sym_n$. If $a\cdot s_n \in (\Z/2)^n \rtimes \Sym_n$ with  $a \in (\Z/2)^n$,
 then the $B_n$-conjugacy class of $a\cdot s_n$ is determined by ${\rm tr}(a) \in \Z/2$ where $a\rightarrow {\rm tr}(a)$
 is the homomorphism ${\rm tr} : (\Z/2)^n \rightarrow \Z/2$ which is adding-up of the co-ordinates of $a$. Thus one may call,
 the  two $B_n$-conjugacy classes among elements of $B_n$ which
 project to an $n$-cycle in $\Sym_n$ as positive $n$-cycle and negative $n$-cycle depending on whether ${\rm tr}(a)=0$ or  ${\rm tr}(a)=1$. These two conjugacy classes
 can also be differentiated by the order of their elements which is $n$ for a positive $n$-cycle, and is $2n$ for a negative $n$-cycle.

 It follows that  conjugacy classes in  $B_n = (\Z/2)^n \rtimes \Sym_n$ are in bijective correspondence with (ordered pairs of) partitions $\{\p_0, \p_1\}$
of $n$ (to be called a bi-partition, or a pair of partitions of $n$), i.e.,  with  $|\p_0| + |\p_1|=n$, with positive $p_i^0$ cycles and negative $p_i^1$ cycles where $p_i^0$ (resp. $p_i^1$) are the parts of the   partition
$\p_0$ (resp. $\p_1$).

The natural map from the set of conjugacy classes in $B_n$ to the set of conjugacy classes in $\Sym_{2n}$ takes the conjugacy class
corresponding to a bi-partition $\{\p_0, \p_1\}$
of $n$ to the conjugacy class in $\Sym_{2n}$ corresponding to the partition $\{\p_0,\p_0,
2\p_1\}$
of $2n$ where  $2\p_1$ stands for the partition in which each component of
$\p_1$ is multiplied by 2.

As a consequence of the above description of conjugacy classes in $B_n \subset \Sym_{2n}$, observe that the natural mapping from the set of conjugacy classes  in $B_n$ to the set of
conjugacy classes in $\Sym_{2n}$  is not
one-to-one. For example, the conjugacy class of the fixed point free involution $w_0$ in $\Sym_{2n}$
intersects with exactly with
$[n/2]+1$ conjugacy classes in $B_n$ (where $[x]$ denotes the integral part of a real number $x$)
represented by the bi-partitions  $\{\underline{p_d}^+, \underline{p_d}^-\}$ where  $\underline{p_d}^+$ is the partition $ \{2,2,\cdots, 2\}$ of $2d \leq n$, and  $\underline{p_d}^-$ is the trivial partition
$\{1,1,\cdots,1\}$ of $(n-2d)$.

\begin{definition}(Norm mapping) The norm mapping from the set of conjugacy classes in $\Sym_{2n}$ to
  the set of conjugacy classes in $B_n = (\Z/2)^n \rtimes \Sym_n$ is defined on those
   conjugacy classes $g$ in $\Sym_{2n}$ represented only by even cycles $ \{2p_1 \geq 2p_2\geq \cdots \geq 2p_r \geq 2\}$
   (these conjugacy classes in particular have no fixed points under the action of $\Sym_{2n}$ on $\{\pm 1, \pm 2,\cdots, \pm n\}$)
     which are
     mapped to $\Nm(g)=h$, a conjugacy class in  $B_n = (\Z/2)^n \rtimes \Sym_n$ corresponding to the bi-partition $ \p = ( \p_0,  \p_1)$ of $n$ 
     having only positive cycles, i.e., $ \p_1$ is empty, and
     $\p_0 =  \{p_1 \geq p_2\geq \cdots \geq p_r \geq 1\}$.

     The norm mapping from $\Sym_{2n+1}$ to $B_n = (\Z/2)^n \rtimes \Sym_n$ is defined on the set of those
   conjugacy classes $g$ in $\Sym_{2n+1}$ which come from $\Sym_{2n}$, and for them the norm mapping is the one just defined for $\Sym_{2n}$ above.
     
 \end{definition}

 \begin{remark}
   There is another way to define the norm mapping  which relates to more
   conventional way norm mappings are defined in questions about basechange character identities (say, relating twisted characters for $G(\F_{q^d})$ to characters of $G(\F_q)$),
   which we briefly discuss now. 
   Recall that there is a natural mapping from
representations $\pi$ of $\Sym_{2n}$ to representations $\pi'$ of
the semidirect product group $\Sym_{2n}\rtimes \langle w_0\rangle$ in which the underlying
vector space for $\pi'$ is the same as that for $\pi$, and
for which   $w_0\in \langle w_0\rangle$
operates on $\pi'$ as  $w_0$ does on $\pi$. Under this identification
of representations of $\Sym_{2n}$ with one of
$\Sym_{2n}\rtimes \langle w_0\rangle$, the { character}
of $\pi'$ at an element $g \rtimes w_0$ is the same as the character of $\pi$ at the element $gw_0$. In the
basechange character identity, one relates the { character}
of $\pi'$ at an element $g \rtimes w_0$  (called the {\it twisted character} of $\pi'$ at the element $g$)
to the character of a representation of
the subgroup $B_n$ at the element of $B_n$ which is essentially $g\cdot (w_0gw_0) = (gw_0)^2$. Thus, in terms of $\pi$, we will be relating
the character of $\pi$ at $gw_0$ to the character of a representation of $B_n$, call it $\lambda$,
at the element $(gw_0)^2$ which we do not know if it belongs to $B_n$, and even if it does,
it may not define a unique conjugacy class in $H$. From the point of view of $\Sym_{2n}$, the element
$x=gw_0$ is of course any element, and then the question is if there is a good supply of elements
$x \in \Sym_{2n}$ such that $x^2 \in B_n$, and then to find a {\it suitable} conjugacy class in $B_n$ containing
$x^2$. The element $x=1$, and such `singular' elements must be avoided if we are to have any hope of character identity of the representation $\pi$ of $\Sym_{2n}$
at $x$, with the representation $\lambda$ of $B_n$ at the element $\Nm (x)$.

   Since the square
   of a $2d$ cycle is a product of two disjoint cycles of length $d$, it is clear by our earlier description of conjugacy classes in $B_n$ that if  $x\in \Sym_{2n}$ is without fixed points for which all cycles are of even length, then $x^2$ arises from a conjugacy class in $B_n$, although not a unique
   conjugacy class in $B_n$. However, there is
   a unique conjugacy class in $B_n$ consisting only of positive cycles
 such that the intersection
of    the conjugacy class of $x^2$ in $\Sym_{2n}$ with $B_n$ is that conjugacy class in $B_n$,
defining $\Nm(x)$ (on the set of elements $x \in \Sym_{2n}$ which are product of disjoint even cycles
without fixed points).

   \end{remark}

 \section {Basechange  of representations}
 It is well-known that the irreducible representations of $\Sym_{2n}$ are in natural
 bijective correspondence
 with partitions of $2n$. We denote this correspondence by $\p \rightarrow \pi(\p).$ Similarly, by Clifford theory,
 irreducible representations of  $B_n = (\Z/2)^n \rtimes \Sym_n$ are parametrized by an ordered
 pair of partitions  $\{\p_0, \p_1\}$  with  $|\p_0| + |\p_1|=n$.
 Denote the corresponding irreducible  representation of $B_n$ by  $\pi(\p_0,\p_1 ).$

 Before we describe the basechange map from irreducible representations of $B_{n}$ to  irreducible representations of $\Sym_{2n}$, we recall the notion of 2-core partitions, and 2-quotient of a partition.

 A partition is called a 2-core partition if none of the hook lengths in its Young diagram
 is a multiple of 2. (Same definition replacing 2 by any prime $p$ defines a $p$-core partition.)
 It is easy to see that a 2-core partition exists for a number $n$ if and only if the number $n$ is a triangular number:
 $$n = \frac{d(d+1)}{2},$$
 and in this case there is a unique 2-core partition which is the stair-case partition
 $$\{d,d-1,d-2,\cdots, 2,1\}.$$

 Every partition $\p$  has a 2-core partition, call it   $c_2(\p)$
 associated to it, and a pair of partitions   $(\p_0,  \p_1)$
 called the 2-quotient of   $\p$
 such that
 $$ |\p| = |c_2(\p)| +2(|\p_0|  + | \p_1|).$$

 We briefly recall the definitions of 2-core and 2-quotient partitions of any partition,
 more generally with 2 replaced by any prime $p$, associated to a partition. The notion of a $p$-core and $p$-quotient of a partition is due to Littlewood, cf. \cite{Li2}, which arose there in his study of modular representations
 of the symmetric group.

 We will consider a partition interchangeably with the associated Young diagram which has the obvious notion of its rim. For a square in the Young diagram with hook length $p$, $p$ any prime, one looks at its `arm' and `leg', and removes all the squares on the rim lying in-between those squares on the 
 rim where the arm and the leg intersect the rim, and including these  squares  on the 
 rim where the arm and leg intersect the rim. Iterating this process till there is no square
 in the Young diagram of hook length $p$, we come to a Young diagram called the $p$-core of the partition, or of its Young diagram; $p$-core of a partition is a $p$-core partition. Since in each step, one is removing $p$ squares from the Young diagram, the size of the $p$-core is congruent to the size of the original Young diagram modulo $p$.
 
 Here is how one gets the 2-core of a Young diagram. Begin with the bottom most row of the diagram.
 If it has even length, remove it; if it has odd length, remove all boxes but one.
  If there are two consecutive rows (anywhere) which are both of even size or both of odd size, remove these two. At the end, you
 clearly get a Young diagram in which consecutive rows alternate being of
 even and odd size, with the minimal    one of odd size. The 2-core of the partition is
 $(d,d-1,\cdots, 1)$ where $d$ is the number of rows left at the end.

 Now we define the $p$-quotient of a partition   $\p$ which is an ordered collection of
 $p$ partitions   $(\p_0,  \p_1, \cdots \p_{p-1})$.

 Define $\beta$-numbers associated to a partition $\underline{\mu} := \{\mu_1 \geq \mu_2 \geq \cdots \geq \mu_r\}$ (we allow some of the $\mu_i$ to be zero) to be the collection of (now strictly decreasing) numbers $$\{\mu_1+(r-1), \mu_2+(r-2), \cdots, \mu_r \}.$$
 What needs to be kept in mind is that there is an equivalence relation on the set of partitions
 defined by declaring  $\underline{\mu} = \mu_1 \geq \mu_2 \geq \cdots \geq \mu_r$,  and
 $\underline{\nu} = \nu_1 \geq \nu_2 \geq \cdots \geq \nu_s$ to be in the same equivalence class
 if and only if they are the same after adding a few 0's at the tails. This equivalence gives rise to a similar 
 equivalence on associated
 $\beta$-numbers (but which looks quite different); but the point is that from a sequence of $\beta$ numbers (any strictly decreasing sequence of non-negative numbers) one can obtain the
 corresponding  partition un-ambiguously. 

 Coming back to the partition $\p   = p_1 \geq p_2 \geq \cdots \geq p_r$, we
 now assume that $r$ is divisible by $p$ (adding zeroes at the end if
 necessary). For the associated $\beta$-numbers,  $\beta(\p)= \{p_1+(r-1), p_2+(r-2), \cdots, p_r\}$,
 and for each integer $i$, $0 \leq i \leq p-1$, let  $\beta^i(\p)$ be those
 numbers in  $\beta(\p)= \{p_1+(r-1), p_2+(r-2), \cdots, p_r \}$ which are congruent
 to $i$ mod $p$. Define a new sequence of $\beta$-numbers by substracting $i$ from each of the numbers in $\beta^i(\p)$, and then dividing by $p$. These $\beta$ numbers, call them
 $\beta^i(\p)$, are associated to a partition  $\p_i$,
defining  $p$-quotient of    $\p$, an ordered collection of
 $p$ partitions   $(\p_0,  \p_1, \cdots \p_{p-1})$.

It can be seen that to any partition $\p$, its $p$-core $c_p(\p)$,
    and  $p$-quotient $(\p_0,  \p_1, \cdots \p_{p-1})$, determine
    the partition $\p$ uniquely, and conversely, any data of a $p$-core and a $p$-quotient
    is associated to a  partition $\p$. In particular, for $p=2$, and $n$ any integer,
    the set of bi-partitions    $(\p_0,  \p_1)$ of $n$ is in bijective
    correspondence with partitions of $2n$ with empty 2-core, and 2-quotient,  $(\p_0,  \p_1)$.    We will construct the 2-quotient of a 
partition    in some detail now to bring out a few points needed for our work.

    Let $\p = \{p_1 \geq p_2 \geq \cdots \geq p_m \}$ be a partition of $|\p|$ with associated $\beta$-partition:
    $$\beta(\p) := \{p_1+(m-1) >  p_2+(m-2) > \cdots > p_m\}.$$
    Consider the even and odd parts of this $\beta$-partition as:
    \begin{eqnarray*}
      \beta^0(\p):= & & p_{i_1}+(m-i_1)>  p_{i_2}+(m-i_2) > \cdots >  p_{i_k}+(m-i_k) \\
      \beta^1(\p):= & & p_{j_1}+(m-j_1)>  p_{j_2}+(m-j_2) > \cdots >  p_{j_\ell}+(m-j_\ell).
      \end{eqnarray*}

        Divide the numbers appearing in $\beta^0(\p)$ and $\beta^1(\p)-1$ by 2:
        \begin{eqnarray*}
      \frac{\beta^0(\p)}{2}: & & \frac{p_{i_1}+(m-i_1)}{2}>  \frac{p_{i_2}+(m-i_2)}{2} > \cdots >  \frac{p_{i_k}+(m-i_k)}{2} \\
      \frac{\beta^1(\p)-1}{2}: & & \frac{p_{j_1}+(m-j_1-1)}{2}>  \frac{p_{j_2}+(m-j_2-1)}{2} > \cdots >  \frac{p_{j_\ell}+(m-j_\ell-1)}{2}.
      \end{eqnarray*}

        These sequence of strictly decreasing numbers are $\beta$-numbers for the partitions:
        \begin{eqnarray*}
      \p_0: & & \frac{p_{i_1}+(m-i_1)}{2} -(k-1) \geq  \frac{p_{i_2}+(m-i_2)}{2} -(k-2) \geq  \cdots \geq   \frac{p_{i_k}+(m-i_k)}{2} \\
      \p_1: & & \frac{p_{j_1}+(m-j_1-1)}{2} -(\ell -1)>    \cdots >
      \frac{p_{j_\ell}+(m-j_\ell-1)}{2}.
      \end{eqnarray*}

        It follows that (using $m= k+\ell$, and the fact that the set of integers $m-i_\alpha$ and $m-j_\beta$ are exactly a permutation of
        the set of integers $0,1,2,\cdots ,m-1$),
         \begin{eqnarray*}
           |\p| -2(|\p_0| + |\p_1|) & =  & 2(k-1 + k-2 + \cdots + 1 + 0) + 2(\ell -1 + \ell -2 + \cdots +1 +0) \\
           & & -[(m-i_1) + (m-i_2) + \cdots + (m-i_k)] \\ & & - [(m-j_1-1) + (m-j_2-1) + \cdots + (m-j_\ell -1)] \\
           &=& k(k-1) + \ell(\ell-1) - \frac{m(m-1)}{2} + \ell \\
           & = & \frac{(k-\ell)(k-\ell -1)}{2}.
                 \end{eqnarray*}
                   As a consequence, we derive that
\begin{enumerate}
\item  $|\p|  = 2(|\p_0| + |\p_1|)$
  if and only if $k = \ell$, or $k=\ell +1$. 
  Equivalently, the partition $\p$ has empty 2-core
  (thus necessarily with $|\p|$ even)
  if and only if $k = \ell$, or $k=\ell +1$;
  i.e., in the $\beta$  numbers associated with $\p$, either half of them are even and half of them are odd, or the
  even ones are one more than the odd ones.
\item  $|\p|  = 1+ 2(|\p_0| + |\p_1|)$ if and only if $k +1 = \ell$, or $k=\ell +2$.
    Equivalently, the partition $\p$ has 2-core consisting of $1$ (thus necessarily with $|\p|$ odd)
  if and only if  in the $\beta$  numbers associated with $\p$, odd numbers are  one more in cardinality than even numbers, or the even ones are two more than the odd ones.
\end{enumerate}

We summarize the above in the following proposition.

\begin{prop}
\label{prop:2-core}
A partition $\p$ has empty 2-core
  (thus necessarily with $|\p|$ even)
  if and only if $k = \ell$, or $k=\ell +1$,
  i.e., in the $\beta$  numbers associated with $\p$, either half of them are even and half of them are odd, or the
  even ones are one more than the odd ones.

  A partition $\p$ has 2-core consisting of $1$ (thus necessarily with $|\p|$ odd) if and only if $k +1 = \ell$, or
  $k=\ell +2$, i.e.,   in the $\beta$  numbers associated with $\p$, the odd numbers are  one more in cardinality
  than the even numbers, or the even ones are two more than the odd ones.

  In particular, if $|\p|$ is even, and has even number of parts, then $\p$ has empty core if and only if $k=\ell$,
  whereas if $|\p|$ is odd, and has odd number of  parts, then $\p$ has 2-core 1 if and only if $k+1=\ell$. (In our later analysis, we will always assume adding zeros at the end if necessary
  that if $|\p|$ is even,
  it has even number of parts whereas if $|\p|$ is odd, it has odd number of parts.)
  \end{prop}

 \begin{definition} (Basechange) Define the basechange of an
   irreducible  representation   $\pi(\p_0,\p_1 )$ of $B_n$ to be
   the irreducible representation of $\Sym_{2n}$ corresponding to the (unique) partition
   $\p$ such that the 2-core of   $\p$ is empty, and the
   2-quotient of $\p$ is $(\p_0,\p_1)$, in particular, we have
   $ |\p| = 2n= 2(|\p_0|  + | \p_1|).$
Further, define  the basechange of $\pi(\p_0,\p_1 )$ to $\Sym_{2n+1}$ to be
   the irreducible representation of $\Sym_{2n+1}$ corresponding to the (unique) partition
   $\p$ such that the 2-core of   $\p$ is 1, and the
   2-quotient of $\p$ is $(\p_0,\p_1)$, in particular, we have
   $ |\p|= 2n+1 = 1+ 2(|\p_0|  + | \p_1|).$
 \end{definition}

 \section{Example of correspondence between representations of $\Sym_{2n}$ and $\Sym_{2n+1}$}
 In the previous section, we have defined a basechange map from irreducible representations of $B_n$ to irreducible representations
 of $\Sym_{2n}$ as well as $\Sym_{2n+1}$, which induces a bijection between certain irreducible representations  of $\Sym_{2n}$ with certain irreducible representations of $\Sym_{2n+1}$, more precisely there is a bijection between
 irreducible representations  of $\Sym_{2n}$ which correspond to partitions which have empty 2-core with those irreducible representations of $\Sym_{2n+1}$  which correspond to partitions which have 2-core $(1)$, the bijection is defined by demanding the partitions of $2n$ and $2n+1$ have the same 2-quotient.

In this section, as an example, we tabulate the bijection of partitions of 8 with
empty 2-core to partitions of 9 with core (1), inducing a map on representations of the corresponding symmetric groups. We have illustrated this to point out that what is an obvious map from the set of partitions of $2n$ to set of partitions of $2n+1$ from the
point of view of 2-core and 2-quotient, has no obvious formulation from the point of view of the partitions directly. The table below
also gives the character values $\Theta(w_0)$ and  $\Theta'(w_0)$ for the respective representations of $\Sym_8$ and $\Sym_9$, once
again, the character values $\Theta(w_0)$ and  $\Theta'(w_0)$ which are nonzero are the same up to a sign, but the sign is
not obvious by looking at the table, but for which we did give a recipe earlier.

 \begin{align*}
[ 1^8 ] & \stackrel{ \Theta(w_0)=  \Theta'(w_0)= 1}{\xrightarrow  {\hspace*{35mm}}  }  [ 1^9 ]\\
[ 2,1^6 ] & \stackrel{\Theta(w_0) = -\Theta'(w_0)=-1}{  \xrightarrow  {\hspace*{35mm}} }  [ 3, 2, 1^4 ]\\
[ 2^2, 1^4 ] &\stackrel{\Theta(w_0)= - \Theta'(w_0)= 4}{ \xrightarrow  {\hspace*{35mm}} }  [ 3, 1^6 ]\\
[ 2^3, 1^2 ] & \stackrel{\Theta(w_0)=-\Theta'(w_0)= -4}{\xrightarrow  {\hspace*{35mm}}  } [ 3, 2^3 ]\\
[ 2^4 ] & \stackrel{\Theta(w_0)= -\Theta'(w_0)=6}{\xrightarrow  {\hspace*{35mm}} }  [ 3, 2^2, 1^2 ]\\
[ 3, 1^5 ] &\stackrel{\Theta(w_0)=-\Theta'(w_0)= -3}{ \xrightarrow  {\hspace*{35mm}} }  [ 2^2, 1^5 ]\\
[ 3, 2^2, 1 ] &\stackrel{\Theta(w_0) =-\Theta'(w_0)= -2}{ \xrightarrow  {\hspace*{35mm}} }  [ 2^4, 1 ]\\
[ 3^2, 1^2 ] &\stackrel{\Theta(w_0)=\Theta'(w_0)=8}{ \xrightarrow  {\hspace*{35mm}} }  [ 3^2, 1^3 ]\\
[ 3^2, 2 ] &\stackrel{\Theta(w_0)= -\Theta'(w_0)=-6}{ \xrightarrow  {\hspace*{35mm}} }  [ 3^3]\\
[ 4, 1^4 ] & \stackrel{\Theta(w_0)=-\Theta'(w_0)= 3}{\xrightarrow  {\hspace*{35mm}} }  [ 5, 2, 1^2 ]\\
[ 4, 2, 1^2 ] & \stackrel{\Theta(w_0)=-\Theta'(w_0)= -6}{\xrightarrow  {\hspace*{35mm}} }  [ 5, 1^4 ]\\
[ 4, 2^2 ] & \stackrel{\Theta(w_0)=\Theta'(w_0)= 8}{\xrightarrow  {\hspace*{35mm}} }  [ 5, 2^2]\\
[ 4, 3, 1 ] &\stackrel{\Theta(w_0)=-\Theta'(w_0)=-2}{ \xrightarrow  {\hspace*{35mm}} }  [ 5, 4 ]\\
[ 4^2 ] & \stackrel{\Theta(w_0)=-\Theta'(w_0)=6}{\xrightarrow  {\hspace*{35mm}} }  [ 5, 3, 1 ]\\
[ 5, 1^3 ] &\stackrel{\Theta(w_0)=-\Theta'(w_0)=3}{ \xrightarrow  {\hspace*{35mm}} }  [ 4, 2, 1^3 ]\\
[ 5, 3 ] &\stackrel{\Theta(w_0)=-\Theta'(w_0)=-4}{ \xrightarrow  {\hspace*{35mm}} }  [ 4^2, 1 ]\\
[ 6, 1^2 ] &\stackrel{\Theta(w_0)=-\Theta'(w_0)=-3}{ \xrightarrow  {\hspace*{35mm}} }  [ 7, 2 ]\\
[ 6, 2 ] &\stackrel{\Theta(w_0)=-\Theta'(w_0)=4}{ \xrightarrow  {\hspace*{35mm}} }  [ 7, 1^2 ]\\
[ 7, 1 ] & \stackrel{\Theta(w_0)=-\Theta'(w_0)=-1}{\xrightarrow  {\hspace*{35mm}}  } [ 6, 2, 1 ]\\
[ 8 ] &\stackrel{\Theta(w_0)= \Theta'(w_0)=1}{ \xrightarrow  {\hspace*{35mm}} }  [ 9 ]\\
\end{align*}

 The above table misses out the partitions $[5,2,1]$ and $[3,2,1^3]$ of 8, and partitions $[8,1], [6,3], [6,1^3], [4,3,2],
 [4,3,1^2], [4,2^2,1], [3^2,2,1], [2,1^7], [2^3,1^3], [4,1^5]$ of 9, where $\Theta(w_0) =\Theta'(w_0) = 0$.

 \section{A theorem of Littlewood}
 This section aims to give a proof of a theorem due to Littlewood, see pages 143-146 of \cite{Li1}, on character values
 for symmetric groups at exactly the same set of conjugacy classes that we have considered in this paper: 
 product of disjoint
 even cycles without fixed points --- Littlewood's theorem covers only the case of $\Sym_{2n}$, which we prove also for $\Sym_{2n+1}$.
 Our main theorem proved in the next section is a simple consequence of the theorem of Littlewood
 (suitably extended to $\Sym_{2n+1}$).
 
 We have decided to include a proof of the theorem due to Littlewood since the proof (from 1940!) is hard to follow.
 All the proofs
on character values of representations of the symmetric group eventually boil down to a theorem of Frobenius, which is a consequence
of Schur-Weyl duality, turning theorems about character theory of $\GL_n(\C)$ to character theory for the symmetric group. We begin by
recalling the relationship between the two character theories: it is remarkable that they are the same!

Let $R_d$ be the representation ring of the symmetric group $\Sym_d$, treated here only as an abelian group.
Let $R= \sum_{d=0}^{\infty} R_d$, together with the multiplication  $ R_n \otimes R_m \rightarrow R_{n+m}$ which corresponds to induction
of a representation $(V_1 \boxtimes V_2)$ of $\Sym_n \times \Sym_m$ to $\Sym_{n+m}$. These multiplications turn $R$ into a commutative and associative
graded ring. It can be seen that as graded rings, 
$ R \cong \Z[H_1, \cdots, H_d,\cdots ]$, the polynomial ring in infinitely many variables $H_i, i \geq 1$, where each $H_i$ is given
weight $i$, and corresponds to the trivial representation of $\Sym_i$.

On the other hand, let
$$\Lambda_n= \Z[X_1, X_2, \cdots, X_n]^{\Sym_n} = \bigoplus_{k\geq 0}  \Lambda_n^k,$$
where $\Lambda_n^k$ is the space of symmetric polynomials in $ \Z[X_1, X_2, \cdots, X_n]$ of degree $k$.  

Define, $$\Lambda^k = \varprojlim{\Lambda_n^k} ,$$
where the inverse limit is taken with respect to natural map of polynomial rings
$$ \Z[X_1, X_2, \cdots, X_{n+1} ]^{\Sym_{n+1}} \rightarrow \Z[X_1, X_2, \cdots, X_n ]^{\Sym_n}$$
in which $X_{n+1}$ is sent to the zero element.

Finally, define the graded ring
$$\Lambda = \bigoplus_{k \geq 0} \Lambda^k.$$
The ring $\Lambda$, often by abuse of language (in which we too will indulge in) is called the ring of symmetric polynomials in infinitely many variables, comes equipped with surjective homomorphisms to $\Lambda_n$ for all $n \geq 0$, which is in fact an isomorphism
restricted to $\Lambda^k$ onto $\Lambda^k_n$ for $k \leq n$.

As typical elements of the ring $\Lambda$, also of paramount importance, note the following symmetric functions of degree $m$ in infinitely many variables $X_1,X_2,X_3,\cdots$:

\begin{enumerate}
\item $p_m=X_1^m+X_2^m+\cdots$ (an infinite sum),
\item $e_m = \sum {X_{i_1} X_{i_2}\cdots X_{i_m}},$ where the sum is over all indices $i_1<i_2 <\cdots< i_m$,
  \item $h_m=  \sum {X_{i_1} X_{i_2}\cdots X_{i_m}},$ where the sum is over all indices $i_1\leq i_2 \leq \cdots \leq  i_m$.
    \end{enumerate}

Later, for any partition $\underline{\lambda} = \{\lambda_1 \geq \lambda_2 \geq \cdots \geq \lambda_m \}$ of $k$, we will have occasion
to use the following symmetric functions (in infinite number of variables) of degree $k$:

\begin{enumerate}
\item $p_{\underline{\lambda}} =p_{\lambda_1} \cdot p_{\lambda_2} \cdots p_{\lambda_m}$, 
\item $e_{\underline{\lambda}} =e_{\lambda_1} \cdot e_{\lambda_2} \cdots e_{\lambda_m}$,
  \item $h_{\underline{\lambda}} =h_{\lambda_1} \cdot h_{\lambda_2} \cdots h_{\lambda_m}$.
      \end{enumerate}

As $\underline{\lambda}$ varies over all partitions of $k$, $e_{\underline{\lambda}}$ form a basis of the space of symmetric polynomials (in infinitely many variables) of degree $ k$; similarly,
$h_{\underline{\lambda}}$ form a basis of the space of symmetric polynomials of degree $ k$; whereas  $p_{\underline{\lambda}}$ forms a basis after $\Z$ is replaced by $\Q$ as the coefficient ring.

Observe that the graded rings  $R= \sum_{d=0}^{\infty} R_d$ and $\Lambda= \sum _{k=0}^{\infty} \Lambda^k$ are
isomorphic under the map which sends $H_m$ to $h_m$, call this map $\Psi$.

Note that a polynomial representation of $\GL_m(\C)$ has as its character at the diagonal element $(X_1, X_2, \cdots, X_m)$ in  $\GL_m(\C)$,
a symmetric polynomial in $\Z[X_1, \cdots,  X_m]^{\Sym_m}$.
A nice fact about irreducible polynomial representations  of $\GL_m(\C)$, say with highest weight
$\lambda_1 \geq \lambda_2 \geq \cdots \geq \lambda_m$ is that it makes sense to speak of ``corresponding irreducible representations'' 
of $\GL_d(\C)$ for {\it all} $d \geq m$  by extending the highest weight
$\lambda_1 \geq \lambda_2 \geq \cdots \geq \lambda_m$
by adding a few zeros after $\lambda_m$. 
The characters of these 
representations of $\GL_d(\C)$ for $d \geq m$, which are symmetric polynomials in $ \Z[X_1, X_2, \cdots X_{d} ]^{\Sym_{d}}$,  
correspond to each other under the maps:
$ \Z[X_1, X_2, \cdots X_{d} ]^{\Sym_{d}} \rightarrow \Z[X_1, X_2, \cdots X_{d'} ]^{\Sym_{d'}}$ for $d\geq d' \geq m$. Therefore, a polynomial
representation of $\GL_m(\C)$ has as its character an element which can be considered to belong to
$\Lambda= \sum _{k=0}^{\infty} \Lambda^k $ (and not only in $\Z[X_1, \cdots,  X_m]^{\Sym_m}$).
For example, the character of the standard
$m$-dimensional representation  of $\GL_m(\C)$ is the infinite sum $X_1+X_2+ \cdots$.

So far, nothing non-obvious has been said. Now, here is a non-obvious fact, a form of the Schur-Weyl duality, that if
$\underline{\lambda} = \{\lambda_1 \geq \lambda_2 \geq \cdots \geq \lambda_m \}$ is a
partition of $n$, defining an irreducible representation of $\Sym_n$, say $\pi_{\underline{\lambda}}$, and defining at the same time an irreducible
representation of $\GL_d(\C)$ with highest weight $\{\lambda_1 \geq \lambda_2 \geq \cdots \geq \lambda_m \geq 0 \geq \cdots \geq 0\}$
for all $d \geq m$, with its character, the Schur function
${\mathcal S}_{\underline{\lambda}}$,
then $$\Psi(\pi_{\underline{\lambda}}) = {\mathcal S}_{\underline{\lambda}}.$$
(It is helpful to recall that under the Schur-Weyl duality, the trivial representation of ${\mathcal S}_m$ goes to the irreducible representation
${\rm Sym}^m(\C^d)$ of $\GL_d(\C)$ corresponding to the partition $\{1,1,\cdots, 1\}$ of $m$.)

A form of the Frobenius theorem on characters of the symmetric group 
will play an
important role in this work, so we recall this as well as supply a proof in the next Proposition. (See for example
\cite{F-H}, Exercise 4.52(e), for the statement of this result.)

\begin{prop} \label{Frob-ch}
For $\underline{\lambda} = \{\lambda_1 \geq \lambda_2 \geq \cdots \geq \lambda_m \}$,
a partition of $n$, defining an irreducible representation $\pi_{\underline{\lambda}}$ of $\, \Sym_n$,  
and defining at the same time an irreducible
representation of $\GL_d(\C)$ with highest weight $\{\lambda_1 \geq \lambda_2 \geq \cdots \geq \lambda_m \geq 0 \geq \cdots \geq 0\}$
for all $d \geq m$, with its character, the Schur function
${\mathcal S}_{\underline{\lambda}}$, we have the following identity:
\[ \tag{1} {\mathcal S}_{\underline{\lambda}} = \sum _{\rho} \frac{\Theta_{\underline{\lambda}} (c_\rho)}{|Z(c_\rho)|} p_\rho ,\]
where $\rho = (\rho_1,\rho_2, \cdots )$ is a partition of $n$, defining a conjugacy class $c_\rho$ in $\Sym_n$, whose centralizer
in $\Sym_n$ is $Z(c_\rho)$ of order  $|Z(c_\rho)|$; the symmetric function $p_\rho $ is the product of symmetric polynomials
$p_{\rho_i}= X_1^{\rho_i}+X_2^{\rho_i} + X_3^{\rho_i}+ \cdots$.
\end{prop}

\begin{proof}
  We will prove a generalised version of the identity being proved:
  \[ \tag{1} {\mathcal S}_{\underline{\lambda}} = \sum _{\rho} \frac{\Theta_{\underline{\lambda}} (c_\rho)}{|Z(c_\rho)|} p_\rho ,\]
  to the following identity among symmetric functions:
\[ \tag{2} {\Psi(\pi)} = \sum _{\rho} \frac{\Theta_{\pi} (c_\rho)}{|Z(c_\rho)|} p_\rho ,\]
where $\pi$ is any representation of a symmetric group $\Sym_n$, and $\Psi(\pi)$ is the associated symmetric function of degree $n$ in infinitely many
  variables arising through the isomorphism $\Psi: R \rightarrow \Lambda$. 
  (The symmetric function $\Psi(\pi)$ can of course be identified to the character of a representation of  $\GL_m(\C)$.)
    Once (2) is proved, appealing to the Schur-Weyl duality according to which the isomorphism of graded rings $\Psi: R \rightarrow \Lambda$
  which is defined by sending the trivial representation $H_n$ of $\Sym_n$ to the symmetric polynomial $h_n$ (in infinitely many variables)
  takes the representation $\pi_{\underline{\lambda}}$ to the Schur function  ${\mathcal S}_{\underline{\lambda}}$ proves (1) and hence the Proposition.

  Since the identity (2) is linear in the representation $\pi$, it suffices to check it on a set of linear generators (over $\Z$, although it suffices
  to do it for generators over $\Q$ too) of the ring $R$. Our proof of identity (2) will therefore be accomplished in two steps.

  \begin{enumerate}
  \item Prove that the identity (2) holds for  the trivial representation $H_n$ of ${\Sym}_n$.
  \item If the identity (2) holds for a representation $\pi_1$ of $\Sym_m$ and $\pi_2$ of $\Sym_n$, then it also holds good for
    the representation $\pi_1 \times \pi_2$ of $\Sym_{m+n}$ induced from the representation $\pi_1 \otimes \pi_2$ of
    $\Sym_m \times \Sym_n \subset \Sym_{m+n}$.
  \end{enumerate}

  We begin by proving step 1, i.e. that the identity (2) holds for  the trivial representation $H_n$ of ${\Sym}_n$.
  
  For any integer $m$, let $V_m = \C^m$ be the $m$-dimensional vector space over $\C$. By definition, $h_k$ is the character
  of the representation ${\rm Sym}^k(V_m)$ assuming that $m \geq k$; in fact the integer $m$ will play no role as long as it is
  large enough, preferably infinity! Thus,
  \[ \sum_{n=0}^{\infty} t^n h_n = \frac{1}{(1-tX_1) (1-tX_2) \cdots (1-tX_m)}.\]

  Taking logarithm  on the two sides,
  \begin{eqnarray*}  \log( \sum_{n=0}^{\infty} t^n h_n)  & = & - \sum_{i=1}^{m} \log (1-tX_i), \\
                                                     & = & t(\sum X_i) + \frac{t^2}{2}( \sum X_i^2) +  \frac{t^3}{3}( \sum X_i^3) + \cdots
    \end{eqnarray*}
  Now taking exponential of the two sides:
  \begin{eqnarray*}
    \sum_{n=0}^{\infty} t^n h_n   & = & \exp^{t(\sum X_i)}  \cdot \exp^ {\frac{t^2}{2}( \sum X_i^2) }  \cdot
    \exp^{\frac{t^3}{3}( \sum X_i^3)}  \cdots \\
    & = & [1 +tp_1 + \frac{(tp_1)^2}{2!} + \cdots]
        [  1 + t^2p_2/2+\frac{(t^2p_2/2)^2}{2!} +\cdots][  1 + t^3p_3/3 +\frac{(t^3p_3/3)^2}{2!}+ \cdots]    \end{eqnarray*}

Therefore,

\begin{eqnarray*}
h_n & = & \sum \frac{p_1^{i_1}}{i_1!}\frac{(p_2/2)^{i_2}}{i_2!}\frac{(p_3/3)^{i_3}}{i_3!} \cdots \\
   & \stackrel{*}{=} & \sum \frac{p_1^{i_1}\cdot  p_2^{i_2} \cdot p_3^{i_3} \cdots } {i_1! 2^{i_2} i_2! 3^{i_3} i_3! \cdots}  \\
\end{eqnarray*}
where the summation is taken over $i_1+2i_2+ \cdots =n$. This proves the identity (2) for the trivial representation $H_n$ of $\Sym_n$
for which $\Theta_\pi(c_\rho)$ is identically 1, and the denominators in the right hand side of the equality $(*)$ are the order
of the centralizer of the conjugacy class $c_\rho \in \Sym_n$.

Next, we prove that if the identity (2) holds for a representation $\pi_1$ of $\Sym_m$ and $\pi_2$ of $\Sym_n$, then it also holds good for
the representation $\pi_1 \times \pi_2$ of $\Sym_{m+n}$
induced from the representation $\pi_1 \otimes \pi_2$ of $\Sym_m \times \Sym_n$.

For this, we slightly rewrite the identity (2) as:     
\[ \tag{3} {\Psi(\pi)} = \frac{1}{n!} \sum _{x \in \Sym_n} {\Theta_{\pi} (x)} p_x ,\]
where $\pi$ is a representation of the symmetric group $\Sym_n$ with $\Theta_{\pi}(x)$ its character at an element $x \in \Sym_n$, and $p_x$
is what was earlier denoted as $p_{\rho(x)}$ where $\rho(x)$ denotes the partition of $n$ associated to $x$.

Because $$\Psi(\pi_1 \times \pi_2) = \Psi(\pi_1) \cdot \Psi(\pi_2),$$
assuming that (3) holds for $\pi_1$ and $\pi_2$, to prove that it also holds for $\pi_1 \times \pi_2$,
we are reduced to proving:

\[\frac{1}{(n+m)!} \sum_{g \in \Sym_{m+n}} \Theta_{\pi_1 \times \pi_2}(g)p_g = \frac{1}{m! n!} \sum_{(h_1,h_2) \in \Sym_m \times \Sym_n} \Theta_{\pi_1}(h_1)
\Theta_{ \pi_2}(h_2) p_{h_1} p_{h_2}. \tag{4} \]
This will be a simple consequence of the character of the induced
representation 
$\pi_1 \times \pi_2$ of $\Sym_{m+n}$ given by:
$ \pi_1 \times \pi_2 = {\rm Ind}_{ \Sym_m \times \Sym_n}^{\Sym_{m+n}} (\pi_1 \otimes \pi_2)$.

Note the well-known identity regarding the character $f'$ of the induced representation ${\rm Ind}_H^G(U)$
of a representation $U$ of $H$ with character $f$ (a function on $H$, extended to a function
on $G$ by declaring it zero outside $H$) at an element $s \in G$:
$$f'(s) = \frac{1}{|H|} \sum_{t \in G} f(t^{-1}st).$$
It follows that for any conjugacy class function $\lambda(s)$ on $G$,
\[\frac{1}{|G|} \sum_{s \in G} f'(s)\lambda(s)  = \frac{1}{|H|} \sum_{t \in H} f(t) \lambda(t). \tag{5}\]
Now (3) follows from applying (5) to the induced representation
$ \pi_1 \times \pi_2 = {\rm Ind}_{ \Sym_m \times \Sym_n}^{\Sym_{m+n}} (\pi_1 \otimes \pi_2)$, and where $\lambda(s) = p_s$
for $s\in \Sym_{m+n}$ are symmetric functions used before (products of $X_1^{i_1}+X_2^{i_1}+X_3^{i_1} + \cdots$). (We
are thus applying the identity (5) for $\lambda(s)$ not a scalar valued function on $G$, but rather with values in the
ring of symmetric polynomials, which we leave to the reader to think about!)

We have thus completed the proof of the Proposition.\end{proof}

Before proceeding further,  we need to introduce a sign $\epsilon(\p)$ associated to any partition with either empty
2-core or with 2-core 1.

\begin{definition}  \label{shuffle} (Sign of a partition)
  (a)
  For a partition $\p$ of an even number with even number of parts $2m$ (by adding a zero at the end if necessary)
  with $\beta$-numbers
   $\beta(\p) =\{\beta_{2m-1} > \beta_{2m-2} > \cdots > \beta_1 >\beta_0\}$, and
  with   the  even and odd $\beta$-numbers: $\beta^0(\p)$ and  $\beta^1(\p)$ of the same cardinality.
  Let $X_{2m}=\{2m-1, 2m-2, \cdots, 1, 0\}$, and let $i(\p)$ be the bijection from $X_{2m}$ to $\beta(\p)$ sending $i \rightarrow \beta_i$.
Let  $j(\p)$ be  the unique order preserving bijective
  map from the ordered set $\beta^0(\p) \cup \beta^1(\p)$ (in which there is no order relationship between odd and even numbers)
  to the ordered set $X_{2m}$ taking even numbers to even numbers, and odd numbers to odd numbers. 
 The permutation $s(\p)$  on $X_{2m}$ defined as the composition of the maps $X_{2m} \stackrel{i(\p)}{\rightarrow} \beta(\p) \stackrel{j(\p)}{\rightarrow} X_{2m}$ will be called the {\it shuffle permutation} associated to $\p$, and its sign $(-1)^{s(\p)}$
 will be denoted $\epsilon(\p)$.  

 (b) For a partition $\p$ of an odd number with odd number of parts $2m+1$ (by adding a zero at the end if necessary)
  with $\beta$-numbers
   $\beta(\p) =\{\beta_{2m} > \beta_{2m-1} > \cdots > \beta_1 >\beta_0\}$, and
  with   the  even and odd $\beta$-numbers: $\beta^0(\p)$ and  $\beta^1(\p)$ with the number of odd $\beta$-numbers $m+1$,
  let $X_{2m+1}=\{2m+1, 2m, 2m-1, \cdots, 1\}$ (let's be reminded that $X_{2m+1}$ is different from $X_{2m}$, and in particular,
  it has more odd numbers than even numbers).
  Let $i(\p)$ be the bijection from $X_{2m+1}$ to $\beta(\p)$ sending $i \rightarrow \beta_{i-1}$.
Let  $j(\p)$ be  the unique order preserving bijective
  map from the ordered set $\beta^0(\p) \cup \beta^1(\p)$ (in which there is no order relationship between odd and even numbers)
  to the ordered set $X_{2m+1}$ taking even numbers to even numbers, and odd numbers to odd numbers. 
 The permutation $s(\p)$  on $X_{2m+1}$ defined as the composition of the maps $X_{2m+1} \stackrel{i(\p)}{\rightarrow} \beta(\p) \stackrel{j(\p)}{\rightarrow} X_{2m}$ will be called the {\it shuffle permutation} associated to $\p$. Define,
 $\epsilon(\p) = (-1)^m (-1)^{s(\p)}$.  
\end{definition}

\begin{remark} The sign defined here, see Proposition \ref{sign2} below for a much simpler definition, will eventully be the sign of the associated
  respresentation of the symmetric group $\Sym_{|\p|}$ at the element $w_0 =(12)(34)\cdots (2n-1,2n)$ inside $\Sym_{2n}$ or $\Sym_{2n+1}$
  for $|\p| = 2n,2n+1$.
\end{remark}
The proof of the following proposition will be given by Arvind Ayyer in Appendix~.

\begin{prop} \label{sign2}
For a partition $\p$ of an integer $n$,  let $2k$ or $2k+1$,
respectively, be the number of odd entries in $\p$. Associated to the partition $\p$, we define another sign
$\epsilon'(\p) := (-1)^k$. 
If the partition $\p$ corresponds to an irreducible representation $\pi$ of $\Sym_{|\p|}$, then we define $\epsilon'(\pi) =
\epsilon'(\p)$. The the following holds:

$$\epsilon(\p) =\epsilon'(\p).$$
\end{prop}

Next, we recall the following theorem from  \cite{Li1}
where it is contained in equations 7.3.1 and 7.3.2 of page 132; it is also contained in \cite{P} as Theorem 2.
Both \cite{Li1} as well as \cite{P} deal with a more general theorem involving $\GL_{mn}(\C)$. This theorem as well as the next theorem
should be considered as the $\GL_n(\C)$ analogues of theorems on character values that we strive to prove here. It is in this theorem that the notion of $p$-core and $p$-quotients make their appearance (stated here only for $p=2$, although stated more generally in \cite{P}).

\begin{thm} \label{P-L}Let $\underline{\lambda} = \{\lambda_1 \geq \lambda_2 \geq \cdots \geq \lambda_{2m} \}$ be the highest weight of an
  irreducible polynomial representation $V_{\underline{\lambda}}$ of $\GL_{2m}(\C)$ with character  ${\mathcal S}_{\underline{\lambda}}$.
  Let $$X=(x_1,x_2,\cdots, x_m, -x_1, -x_2, \cdots, -x_m) =  (\underline{X}, -\underline{X}), $$ be a diagonal matrix in $\GL_{2m}(\C)$
  with $x_i \in \C^\times$ arbitrary. Then if  ${\mathcal S}_{\underline{\lambda}}(X)$ is not identically zero, its 2-core $c_2(\underline{\lambda})$ must be
  empty, in which case, if 
  $\{  \underline{\lambda}^0, \underline{\lambda}^1 \}$ is the 2-quotient of $\underline{\lambda} $, then,
  $${\mathcal S}_{\underline{\lambda}}(X) = \epsilon ( \underline{\lambda}) {\mathcal S}_{\underline{\lambda}^0}(\underline{X}^2) {\mathcal S}_{\underline{\lambda}^1}(\underline{X}^2),$$
  where $   {\mathcal S}_{\underline{\lambda}^0}$ and ${\mathcal S}_{\underline{\lambda}^1}$ are the characters of the corresponding
  highest weight modules of $\GL_m(\C)$, $ \underline{X}$ is the diagonal matrix  $(x_1,x_2,\cdots, x_m)$, and $\underline{X}^2$
  its square.
  \end{thm}

\begin{proof}
  For the sake of completeness, we give the proof. Write the matrix whose determinant represents the numerator of the
Weyl character formula as (where $\beta_i = \lambda_i+ 2m-i$):

$$
\left ( \begin{array}{ccccccc} 
x_1^{\beta_1} &  x_2^{\beta_1}   & \cdots  & x_m^{\beta_1} & 
(-x_1)^{\beta_1}  &  \cdots &   (-x_m)^{\beta_1}  \\
x_1^{\beta_2} &  x_2^{\beta_2}  &  \cdots & x_m^{\beta_2} & 
(-x_1)^{\beta_2}    & \cdots &   (-x_m)^{\beta_2} \\
\cdot  &  \cdot   &   \cdots   &  \cdot  
&  \cdot  &  \cdots  & \cdot     \\
\cdot  &  \cdot   &   \cdots   &  \cdot  
&  \cdot  &  \cdots & \cdot   \\
x_1^{\beta_{2m}} &  x_2^{\beta_{2m}}   & \cdots &  x_m^{\beta_{2m}} & 
(-x_1)^{\beta_{2m}}  & \cdots &  (-x_m)^{\beta_{2m}} 
\end{array}\right)\cdot$$

\vspace{3mm}

In this $2m\times 2m$-matrix, adding the first $m$ columns of the matrix to the last $m$ columns, we find that all the rows
of the last $m$ columns with $\beta_i$ odd become zero, and those rows with $\beta_i$ even get multiplied by 2. In the new matrix,
substracting the half of last $m$ columns to the first $m$ columns, makes all rows in the first $m$ columns with $\beta_i$ even to be zero. Let $d$ be the number of $\beta_i$ which are odd, and therefore $2m-d$ is the number of $\beta_i$ which are even.
Thus we get a matrix in which in the first $m$ columns,   there are exactly $d$ nonzero rows, and in the last $m$ columns, there
are exactly $2m-d$ complementary rows which are nonzero.

The determinant of such a matrix is nonzero only for $d=2m-d$, i.e., $d=m$, and in which case, by permuting rows, we come to a
block diagonal matrix which looks like (where $\gamma_k, \delta_\ell$ are the odd and even numbers among $\beta_i$ written in decreasing order):

$$
\left ( \begin{array}{cllllllc} 
x_1^{\gamma_1} &  x_2^{\gamma_1}   & \cdots  & x_m^{\gamma_1} & 0 & 0 &  \cdots &   0 \\
x_1^{\gamma_2} &  x_2^{\gamma_2}  &  \cdots & x_m^{\gamma_2} & 0  & 0  & \cdots &   0 \\
\cdot  &  \cdot   &   \cdots   &  \cdot  & \cdot  & \cdot  &  \cdots & \cdot   \\
x_1^{\gamma_{m}} &  x_2^{\gamma_{m}}   & \cdots &  x_m^{\gamma_{m}} & 0& 0   & \cdots & 0 \\
0& 0   & \cdots & 0 & x_1^{\delta_1} &  x_2^{\delta_1}   & \cdots  & x_m^{\delta_1} \\
0& 0   & \cdots & 0 & x_1^{\delta_2} &  x_2^{\delta_2}  &  \cdots & x_m^{\delta_2} \\
\cdot  &  \cdot   &   \cdots   &  \cdot  & \cdot &  \cdot  &  \cdots & \cdot   \\
0& 0   & \cdots & 0 & x_1^{\delta_{m}} &  x_2^{\delta_{m}}   & \cdots &  x_m^{\delta_{m}}
\end{array}\right)\cdot$$

The determinant of this matrix is the product of two Weyl numerators for $\GL_m(\C)$, in which $\delta_i$ being even,
we write $x_j^{\delta_i}$ as $(x_j^2)^{\delta_i/2}$, and $\gamma_i$ being odd, we write  $x_j^{\gamma_i}$ as $(x_j^2)^{(\gamma_i-1)/2} \cdot x_j$.
We have a similar factorization of the Weyl denominator.

Next we observe that the deteriminant of a matrix in which the rows have been shuffled using a permutation $\sigma$ of the rows,
changes by multiplication by the sign $(-1)^\sigma$. This needs to be done both for the numerator which involves the `shuffle permutation' introduced earlier in Definition \ref{shuffle}, and the denominator for which the shuffle permutation is identity matrix; to be more precise, the numerator needs the permutation $s(\p)s_0$, and the denominator needs $s_0$ where $s_0$ is the permutation of
$X_{2m} = \{ 2m-1, 2m-2,\cdots, 1,0\}$  sending odd numbers in $X_{2m}$ consecutively to $\{2m-1, 2m-2,\cdots, m\}$ and even numbers
in $X_{2m}$ consecutively to $\{m-1, \cdots, 1,0\}$.

This completes the proof of the theorem.\end{proof}

We will also need the following variant of the previous theorem which is proved as the previous theorem by manipulations
with the explicit character formula for $\GL_n(\C)$ as quotients of two determinants; we will not give a proof of this theorem.

\begin{thm} \label{P2} Let $\underline{\lambda} = \{\lambda_1 \geq \lambda_2 \geq \cdots \geq \lambda_{2m+1} \}$
  be the highest weight of an
  irreducible polynomial representation $V_{\underline{\lambda}}$ of $\GL_{2m+1}(\C)$ with character  ${\mathcal S}_{\underline{\lambda}}$.
  Let $$X=(x_1,x_2,\cdots, x_m, -x_1, -x_2, \cdots, -x_m,x) =  (\underline{X}, -\underline{X}, x), $$ be a diagonal matrix in $\GL_{2m+1}(\C)$
  with $x_i, x \in \C^\times$ arbitrary. Then if  ${\mathcal S}_{\underline{\lambda}}(X)$ is not identically zero, then the number of even entries and number of odd entries in the $\beta$-sequence
$\beta(\underline{\lambda}) = \{\beta_1=\lambda_1+2m > \beta_2 =\lambda_2+2m-1> \cdots >\lambda_{2m+1} \},$
  differ by one. If the number of odd entries in the $\beta$-sequence is one more than the number of even entries, then
  its 2-core  is $\{1\}$,
  whereas   if the number of even entries in the $\beta$-sequence is one more than the number of odd entries, then
its 2-core is empty. In either case, let
  $\{  \underline{\lambda}^0, \underline{\lambda}^1 \}$ be the 2-quotient of $\underline{\lambda} $. Then,

\[ \tag{a} {\mathcal S}_{\underline{\lambda}}(X) =  
  \epsilon ( \underline{\lambda})  \cdot x \cdot
           {\mathcal S}_{\underline{\lambda}^0}(\underline{X}^2) {\mathcal S}_{\underline{\lambda}^1}(\underline{X}^2, x^2) \]
if the number of odd entries in  the $\beta$-sequence is one more than the number of even entries; whereas 
\[ \tag{b} {\mathcal S}_{\underline{\lambda}}(X) =  
  \epsilon ( \underline{\lambda})  
           {\mathcal S}_{\underline{\lambda}^1}(\underline{X}^2) {\mathcal S}_{\underline{\lambda}^0}(\underline{X}^2, x^2) \]
 if the number of even entries in  the $\beta$-sequence is one more than the number of odd entries.
Here $   {\mathcal S}_{\underline{\lambda}^0}$ and ${\mathcal S}_{\underline{\lambda}^1}$ are the characters of the corresponding
  highest weight modules of $\GL_m(\C)$, and $\GL_{m+1}(\C)$ in case (a), and of $\GL_{m+1}(\C)$, and $\GL_{m}(\C)$ in case (b). 
  \end{thm}

The following theorem is due to Littlewood \cite{Li1}, pages 143-146,   for even symmetric groups. Forms of this theorem seem to be available in the literature more generally for arbitrary
wreath product $(\Z/p)^n \rtimes \Sym_n$ contained in $\Sym_{pn}$,
such as in Theorem 4.56 of \cite{Ro} which also follows from  Theorem 2 of  \cite{P}
and the corresponding analogue of Theorem \ref {P2} for general $p$.

\begin{thm} \label{little} Let $\p \rightarrow \pi(\p)$ be the natural correspondence between
  partitions/Young diagram  and irreducible representations of the symmetric group $\Sym_{|\p|}$.  Then
  if the representation $\pi(\p)$ is to have nonzero character at some element $w$ in $\Sym_{|\p|}$
   which is a product of disjoint even cycles together with at most one fixed point,  the 2-core
   of $\p$ must be empty if $|\p|$ is even, and 
   2-core must be 1 if $|\p|$ is odd.
   Further, if $w=w_1\cdots w_d$, product of disjoint cycles $w_i$ of lengths
  $2 \ell_i$ together with a cycle of length $\leq 1$, then,
  $$\Theta(\pi(\p))(w) = \epsilon(\p) \Theta(\pi(\p_0) \times \pi(\p_1))(w'),$$
  where 
  $w' =w_1'\cdots w_d'$ is an element in $\Sym_n$, $n = [|\p|/2]$ which is a product of disjoint cycles of length $\ell_i$, and where we have used the notation  $\pi(\p_0) \times \pi(\p_1)$ for the
  representation (usually reducible) of $\Sym_{n}$
  $$ \pi(\p_0) \times \pi(\p_1) = {\rm Ind}_{\Sym_{|\p_0|} \times \Sym_{|\p_1|}}^{\Sym_n} (\pi(\p_0) \boxtimes \pi(\p_1)).$$
\end{thm}

\begin{proof}

  Recall once again the notation that if
$\underline{\lambda} = \{\lambda_1 \geq \lambda_2 \geq \cdots \geq \lambda_m \}$ is a
  partition of $k$, it defines an irreducible representation of $\Sym_k$, say $\pi_{\underline{\lambda}}$ with character $\Theta_{\underline{\lambda}}$,
  and defines at the same time an irreducible
  representation of $\GL_d(\C)$ for all $d \geq m$ of highest weight
  $\{\lambda_1 \geq \lambda_2 \geq \cdots \geq \lambda_m \geq 0 \geq \cdots \geq 0 \}$,
  with its character, the Schur function
$\mathcal{S}_{\underline{\lambda}}$, and under the isomorphism $\Psi: R \rightarrow \Lambda$,
\[  \Psi(\pi_{\underline{\lambda}}) = {\mathcal S}_{\underline{\lambda}}. \]

By Proposition \ref{Frob-ch},
\[ \tag{1} {\mathcal S}_{\underline{\lambda}} = \sum _{\rho} \frac{\Theta_{\underline{\lambda}} (c_\rho)}{|Z(c_\rho)|} p_\rho ,\]
where $\rho = (\rho_1,\rho_2, \cdots )$ is a partition of $k$, defining a conjugacy class $c_\rho$ in $\Sym_k$, whose centralizer
in $\Sym_k$ is $Z(c_\rho)$ of order  $|Z(c_\rho)|$; the symmetric function $p_\rho $ is the product of symmetric polynomials
$p_{\rho_i}= X_1^{\rho_i}+X_2^{\rho_i} + X_3^{\rho_i}+ \cdots$

In the standard notation, if a conjugacy class $c$ in $\Sym_k$ is $(1^{i_1},2^{i_2}, \cdots, k^{i_k})$, the order of the centralizer of any element in $c$ is:
$$ (i_1)! 2^{i_2} (i_2)! \cdots k^{i_k} (i_k)!.$$

For a partition $\rho = (\rho_1,\rho_2, \cdots )$ of $k$,
we will be using the notation $2\rho$ for the partition   $2\rho = (2\rho_1, 2\rho_2, \cdots )$ of $2k$, defining a conjugacy class
$c_{2\rho}$ in $\Sym_{2k}$ for which 
the order of the centralizer of any element in $c_{2\rho}$ is:
$$  2^{i_1} (i_1)! 4^{i_2} (i_2)! \cdots (2k)^{i_k} (i_k)!,$$
therefore,
\[ \tag{2}  |Z(c_{2\rho})| = 2^p  |Z(c_\rho)|, \]
where $p = i_1+i_2+\cdots + i_k$.

Now we split the proof of the theorem in two cases.

  \noindent{\bf Case 1:} $|\p| = 2n$.  

We will use the identity expressed by equation (1) at the diagonal matrices  in $\GL_{2m}(\C)$ of the form
$X=(x_1,x_2,\cdots, x_m, -x_1, -x_2, \cdots, -x_m) = (\underline{X},-\underline{X})$. 
An important
observation is that  $p_\rho $ which is the product of the symmetric polynomials
$p_{\rho_i}= (X_1^{\rho_i}+X_2^{\rho_i} + X_2^{\rho_i}+ \cdots)$ must be identically zero on such elements unless all the entries in $\rho$ are even.

Using Theorem \ref{P-L}, we write equation (1) as:
\[ \epsilon ( \underline{\lambda})
 {\mathcal S}_{\underline{\lambda}^0}(\underline{X}^2) {\mathcal S}_{\underline{\lambda}^1}(\underline{X}^2)
= \sum _{\rho} \frac{\Theta_{\underline{\lambda}} (c_{2\rho})}{|Z(c_{2\rho})|} p_{2\rho}(X) .\]
Since $p_{2\rho}(X) = 2^p p_{\rho}(\underline{X}^2)$, where $p = i_1+i_2+\cdots + i_n$, we can rewrite this equation using (2) as:
\[ \tag{3} \epsilon ( \underline{\lambda})
 {\mathcal S}_{\underline{\lambda}^0}(\underline{X}) {\mathcal S}_{\underline{\lambda}^1}(\underline{X})
= \sum _{\rho} \frac{\Theta_{\underline{\lambda}} (c_{2\rho})}{|Z(c_{\rho})|} p_{\rho}(\underline{X}) .\]

Since $\Psi: R \rightarrow \Lambda$ is an isomorphism of rings,  the element
${\mathcal S}_{\underline{\lambda}^0}(\underline{X}) {\mathcal S}_{\underline{\lambda}^1}(\underline{X})$ of $\Lambda$ arises from the image under $\Psi$ of
the representation of $\Sym_n$ that we have called $\pi(\p_0) \times \pi(\p_1)$ in the statement of Theorem \ref{little},
therefore, we have,
\[ \tag{4}  \epsilon ( \underline{\lambda})\sum _{\rho} \frac{\Theta_{   \pi(\p_0) \times \pi(\p_1)   } (c_\rho)}{|Z(c_\rho)|} p_\rho(\underline{X}) = \sum _{\rho} \frac{\Theta_{\underline{\lambda}} (c_{2\rho})}{|Z(c_{\rho})|} p_{\rho}(\underline{X}) .\]
(By linearity, the form of Frobenius theorem on characters in Proposition \ref{Frob-ch}
in equation (1) is valid for all representations -- not necesarily irreducible -- of symmetric groups.)

Since the polynomials $ p_{\rho}(\underline{X})$ are linearly independent, we can equate the coefficients  of $ p_{\rho}(\underline{X})$ on the two sides of equation (4) to  prove the theorem when $|\p|$ is even. By Theorem \ref{P-L}, if $\underline{\lambda}$
has non-empty 2-core, then $  {\mathcal S}_{\underline{\lambda}}$ is identically zero on the set of diagonal elements
of the form $X=(x_1,x_2,\cdots, x_m, -x_1, -x_2, \cdots, -x_m) =  (\underline{X}, -\underline{X})$. By the linear independence
of  the polynomials $ p_{\rho}(\underline{X})$, we deduce that $\Theta_{\underline{\lambda}} (c_{2\rho}) \equiv 0$.

  \vspace{5mm}
  
  \noindent{\bf Case 2:} $|\p| = 2n+1$.
  
  In this case, we will use Frobenius character relationship contained in Proposition \ref{Frob-ch}:
  \[ \tag{5} {\mathcal S}_{\underline{\lambda}} = \sum _{\rho} \frac{\Theta_{\underline{\lambda}} (c_\rho)}{|Z(c_\rho)|} p_\rho ,\]
  at diagonal elements of $\GL_{2m+1}(\C)$ of the form:
  $$X=(x_1,x_2,\cdots, x_m, -x_1, -x_2, \cdots, -x_m,x) =  (\underline{X}, -\underline{X}, x), $$
  with $x_i, x \in \C^\times$ arbitrary.
  Observe that if $\ell$ is an odd integer, then for $X$ as above, $p_\ell(X) = x^\ell$,
  whereas for $\ell$ an even integer $p_\ell(X) = 2p_\ell(\underline{X}) + x^{\ell}$. It follows that for any conjugacy class $\rho$ in
  $\Sym_{2n+1}$, $p_\rho$, and hence each term  in the right hand side of equation
  (5) is divisible by $x$, and all terms except those which correspond to those $\rho$ which are product of disjoint even cycles together with exactly
  one fixed point, contribute a term which is divisible by $x$ and no higher power, and all the other terms are divisible by higher
  powers of $x$. Furthermore, in the case $|\p| = 2n+1$, each of the term $p_\rho(X)/x$ is an even function of $x$, hence
  by equation (5),
  ${\mathcal S}_{\underline{\lambda}}(X)$ as a function of $x$ is an odd polynomial function of $x$. Thus, if ${\mathcal S}_{\underline{\lambda}}(X)$ is nonzero, we must be in case (a) of Theorem \ref{P2} (since in case (b) ${\mathcal S}_{\underline{\lambda}}(X)$ is an even function of $x$), therefore  if ${\mathcal S}_{\underline{\lambda}}(X)$ is a nonzero function of $x$,
  $\underline{\lambda}$ must have 2-core $\{1\}$. 

  By Theorem \ref{P2}, the left hand side of the equation (5) is also divisible by $x$ which  after dividing by $x$ gives,
   $\epsilon ( \underline{\lambda}) 
  {\mathcal S}_{\underline{\lambda}^0}(\underline{X}^2) {\mathcal S}_{\underline{\lambda}^1}(\underline{X}^2, x^2).$

  Thus after dividing both the sides of the Frobenius character relationship in equation (5) by $x$, and then putting $x=0$, we are in exactly the same
  situation as in the proof of the theorem for $\Sym_{2n}$, for which we do not need to repeat the previous argument.  
\end{proof}

  \section{The theorem}

The following theorem is the main  result of this paper and  is a simple consequence of  Theorem \ref{little} of the last section.

\begin{thm} \label{main} Let $G$ be either the symmetric group $\Sym_{2n}$ or the symmetric group $\Sym_{2n+1}$, and $H$ the hyperoctahedral group $B_n = (\Z/2)^n \rtimes \Sym_n$,
  sitting naturally inside $G$ as the centralizer of the involution $w_0= (1,-1)(2,-2)\cdots  (n, -n)$. Then there exists an injective map  
    $BC$ from the set of irreducible representations of $B_n$ to the set of
  irreducible representations of $G$, and a map $\Nm: X \rightarrow B_n$ defined and discussed in \S 2
  on a subset $X \subset G $ consisting of those elements in $G$
  which are disjoint product of even cycles and with at most one fixed point  such that,
  (using the notation $\Theta(V)(g)$ for the character of a representation $V$ of a group $G$ at an element $g \in G$)
    $$\Theta(BC(\pi))(w) = \epsilon(BC (\pi))
    \Theta(\pi)(\Nm w),$$
    for any irreducible representation $\pi$ of $B_n$, and $w \in X \subset G$.
        In particular,
    taking $w=w_0$, we have 
    $$\Theta(BC(\pi))(w_0) = \epsilon(BC (\pi))   \Theta(\pi)(1).$$
    
    Thus the character of an irreducible representation of $G$ at the element $w_0$ is, up to a sign, the  dimension of an irreducible representation of
    $B_n = (\Z/2)^n \rtimes \Sym_n$. Further, the image of the map $\pi \rightarrow BC(\pi)$ from the set of irreducible representations of $B_n$ to the set of
    irreducible representations of $G$ consist exactly of those representations of $G$ whose character at $w_0$ is nonzero.
\end{thm}

\begin{proof}
  The proof of this theorem is a simple consequence of  Theorem \ref{little} of the last section.
Recall that    we have used the notation  $\pi(\p_0) \times \pi(\p_1)$ for the
  representation
  $$ \pi(\p_0) \times \pi(\p_1) = {\rm Ind}_{\Sym_{|\p_0|} \times \Sym_{|\p_1|}}^{\Sym_n} (\pi(\p_0) \boxtimes \pi(\p_1)).$$

  Of course the representation $\pi(\p_0) \times \pi(\p_1)$ of $\Sym_n$ is a complicated sum of irreducible representations for which there is the Littlewood-Richardson rule. However, this complication
  has no role to play for us since instead of $\Sym_n$ we are dealing with the larger group,
  $B_n = (\Z/2)^n \rtimes \Sym_n$ and the representation $\pi(\p_0) \times \pi(\p_1)$ of $\Sym_n$ is the restriction
  to $\Sym_n$ of a representation of $B_n$ that basechanges to the representation $\pi$ of $\Sym_{2n}$.

  Observe that the representation of   $B_n = (\Z/2)^n \rtimes \Sym_n$
  that we are dealing with is
  $$ {\rm Ind}_{A}^{B_n} (V),$$
  where $A$ is the subgroup of $B_n$ which is $A= (\Z/2)^n \rtimes ({\Sym_{|\p_0|} \times \Sym_{|\p_1|}})$,
  and $V$ is the representation of $A$ on which
  $$(\Z/2)^{n} = (\Z/2)^{|\p_0| +|\p_1|} = (\Z/2)^{|\p_0| } \times (\Z/2)^{|\p_1|}$$
  acts by the trivial character on the first factor   $(\Z/2)^{|\p_0| }$, and by the non-trivial
  character on each of the $\Z/2$ factors in  $(\Z/2)^{|\p_1|}$; the subgroup ${\Sym_{|\p_0|} \times \Sym_{|\p_1|}}$
  acts by $\pi(\p_0) \boxtimes \pi(\p_1)$.

  By standard application of Mackey theory (restriction to $\Sym_n$ of an induced representation
  of $B_n$), we find that 
  $$ {\rm Res}_{\Sym_n} ({\rm Ind}_{A}^{B_n} (V)) = {\rm Ind}_{\Sym_{|\p_0|} \times \Sym_{|\p_1|}}^{\Sym_n} (\pi(\p_0) \boxtimes \pi(\p_1)) =\pi(\p_0) \times \pi(\p_1) ,$$
proving the theorem.
  \end{proof}

  Here are a few samples that we first did using GAP software~\cite{GAP}.

  \vspace{5mm}
  
  \noindent{\bf Example 1:} The symmetric group $\Sym_{20}$ has 627 irreducible representations, and the character values $\Theta(w_0)$ is:
  \begin{enumerate}
  \item $>0$ for 227 representations,
  \item $<0$ for 254 representations,
  \item 0 for 146 representations.
  \end{enumerate}
  On the other hand, $B_{10}$ has $227+254 = 481$ irreducible representations, of dimensions exactly
  $|\Theta(w_0)|$.

  \vspace{5mm}
  
  \noindent{\bf Example 2:} The symmetric group $\Sym_{21}$ has 792 irreducible representations, and the character values $\Theta(w_0)$ is:
  \begin{enumerate}
  \item $>0$ for 252 representations,
  \item $<0$ for 229 representations,
  \item 0 for 311 representations.
  \end{enumerate}
  On the other hand, $B_{10}$ has $252+229 = 481$ irreducible representations, of dimensions exactly
  $|\Theta(w_0)|$.

\vspace{5mm}

\noindent{\bf Acknowledgement:}
The authors must thank ICTS-TIFR, Bengaluru for the invitation to the program
``Group Algebras, Representations and Computation'' in October 2019 which brought the two authors together to collaborate on this project. 
The second author thanks Dr Ashish Mishra for his tutorials on $p$-core and $p$-quotients, and in particular, for making
the paper \cite{Li2} available to him. The authors thank Arvind Ayyer for supplying the appendix proving Proposition \ref{sign2}.

\newpage 

\appendix 

\section{Proof of Proposition~\ref{sign2} by Arvind Ayyer}
\label{sec:proof} 
%
%
%
%

Let $\la = (\la_1,\dots,\la_{n})$ be a partition of $n$, written as usual with $\la_1 \geq \cdots \geq \la_{n}$. We will think of $\la$ as a partition with $n$ parts by padding with zeros on the right if necessary. 
For example, the set of partitions of $3$ is $\{(3,0,0), (2,1,0), (1,1,1)\}$.
Then, the $\beta$-set of $\la$, is given by $\beta(\la) = (\beta_1, \dots, \beta_{n})$, where $\beta_i = \la_i + n - i$. 
Then we reformulate Proposition~\ref{prop:2-core} as follows.

\begin{prop}
\label{prop:beta-half}
\begin{enumerate}
\item The partition $\la = (\la_1,\dots,\la_{2n})$ of $2n$ has empty $2$-core if and only if there are exactly $n$ odd integers in $\beta(\la)$.
\item The partition $\la = (\la_1,\dots,\la_{2n-1})$ of $2n-1$ has $2$-core equal to $(1)$ if and only if there are exactly $n$ odd integers in $\beta(\la)$.
\end{enumerate}
\end{prop}

Recall that, given a partition $\la$, the {\em shuffle permutation} in Definition~\ref{shuffle} is the element $\pi \in \Sym_{n}$ mapping odd (resp. even) elements of $\beta(\la)$ to odd (resp. even) elements of $(n-1,\dots,1,0)$ preserving their order. (Observe that any totally ordered set of $n$ elements is canonically isomorphic to the set $\{1,2,\cdots ,n\}$, and therefore a bijection between two totally ordered sets of order
$n$  gives rise to an element of $\Sym_n$.)

The sign of the shuffle permutation is denoted $\epsilon(\la) = \sgn(\pi)$. 
It is easy to see that if $\la$ is a partition of $2n$ (resp. $2n-1$), then $\la$ has an even (resp. odd) number of odd parts. 
The restatement of Proposition~\ref{sign2} is then the following.

\begin{prop}
\begin{enumerate}
\item For $\la \vdash 2n$ with empty $2$-core, let $2k$ be the number of odd parts in $\la$. Then $\epsilon(\la) = (-1)^k$.
\item For $\la \vdash 2n-1$ with $2$-core equal to $(1)$, let $2k-1$ be the number of odd parts in $\la$. Then $\epsilon(\la) = (-1)^{k+n-1}$.
\end{enumerate}
\end{prop}

\begin{proof}
We begin with the proof of case (1).
By Proposition~\ref{prop:beta-half}, exactly $n$ of these $\beta_i$'s are odd.
Let the integers $i$ such that $\beta_i$ is odd be labeled $o_1 < \dots < o_n$, and the integers
$i$ such that $\beta_i$ is even be labeled $e_1 < \dots < e_n$.

Before computing the sign of the shuffle permutation, we first compute the sign
$\epsilon_1$ of the permutation that moves $o_1 < \dots < o_n$
to $1< \dots< n$ (preserving their order), and moves $e_1 < \dots < e_n$
to $(n+1)<  \dots<  2n$ (preserving their order). This can be calculated by looking at the corresponding
action on a vector space $V$ of dimension $2n$ with basis vectors $\{v_1,v_2, \cdots, v_{2n}\}$.
Thus we find that $\epsilon_1$  is given by:
\[ v_1 \wedge v_2 \wedge \cdots \wedge v_n \wedge v_{n+1} \wedge \cdots \wedge v_{2n} = \epsilon_1
v_{o_1} \wedge v_{o_2} \wedge \cdots \wedge v_{o_n} \wedge v_{e_{1}} \wedge \cdots \wedge v_{e_{n}},\]
which means that $\epsilon_1$ has the parity of:
\[(n-e_1+1) + (n-e_2+2) + \cdots + (n-e_n+n) = (o_1+ \cdots + o_n) - (1+ \cdots + n).\]
By the computation just done, if $\epsilon_2$ is
the sign of the permutation that moves all the odd integers $ 1 \leq i \leq (2n-1)$
to positions $1, \dots, n$ (preserving their order), and even integers $ 2 \leq i \leq 2n$
to positions $n+1, \dots, 2n$ (preserving their order), then $\epsilon_2$ is given by:
\[ = (o_1+ \cdots +o_n) - (1+ \cdots + n) = (1 + 3 + \cdots + 2n-1) - (1+ \cdots + n).\]
Therefore, we find that $\epsilon(\p) = \epsilon_1 \epsilon_2$ has the same parity as $(o_1+ \cdots + o_n) +n$.  Thus, we have to prove that the parity of $o_1 + \cdots + o_n + n$ is the same as that of $k$.
Now, since $\beta_{o_i} = \la_{o_i} + 2n - o_i$ is odd, so is $\la_{o_i} - o_i$. Thus, the parity of $o_i + 1$ is the same as that of $\la_{o_i}$.
As a result, we have to show that the parity of $\la_{o_1} + \cdots + \la_{o_n}$ is the same as that of $k$. Among all $\la_{o_i}$'s, those which are even clearly do not contribute to the parity, and those which are odd contribute a $1$. We thus have to prove that the cardinality of the set
$S = \{ i \mid \la_{o_i} \text{is odd} \}$ has the same parity as that of $k$.

We will prove something stronger, namely that $|S| = k$. Suppose that $|S| = j$. Partition the set $\{1,\dots,2n\} = O \cup E$, where $O = \{o_1,\dots,o_n\}$. By assumption, $S \subset O$ are the positions where $\la$ takes odd values. 
For convenience, let $\delta_i = 2n-i$ for $1 \leq i \leq 2n$ so that $\beta_i = \la_i + \delta_i$. 
There are exactly $n$ even and $n$ odd $\delta$'s. Since $\beta_i$ for $i \in O$ is odd, 
$\delta_i$'s for $i \in S$ are even and $\delta_i$'s for $i \in O \setminus S$ are odd. Which means that there are $j$ even and $n-j$ odd $\delta_i$'s for $i \in O$.
Consequently, we must have $n-j$ even and $j$ odd $\delta_i$'s for $i \in E$.
Combining all this information, we infer that the number of odd parts in $\la$ is $j$ (from $S$) and $j$ (from $E$, since $\beta_i$'s for $i \in E$ are even).
But we had assumed that $\la$ has $2k$ odd parts, and therefore $j = k$, proving first case of the Proposition.

Now we consider the second case of the Proposition where the length of the partition $\la$ is $2n-1$.
In this case, arguing just as in case (1), we have to prove that the parity of $o_1 + \cdots + o_n$ is the same as that of $k+1$.

Now, since $\beta_{o_i} = \la_{o_i} + 2n - 1 - o_i$ is odd, so is $\la_{o_i} - o_i - 1$. Thus, the parity of $o_i$ is the same as that of $\la_{o_i}$.
As a result, we have to show that the parity of $\la_{o_1} + \cdots + \la_{o_n}$ is the same as that of $k+1$. Among all $\la_{o_i}$'s, those which are even clearly do not contribute to the parity, and those which are odd contribute a $1$. We thus have to prove that the cardinality of the set
$S = \{ i \mid \la_{o_i} \text{is odd} \}$ has the same parity as that of $k + 1$.

We will again prove something stronger, namely that $|S| = k+1$. Suppose that $|S| = j$. Partition the set $\{1,\dots,2n-1\} = O \cup E$, where $O = \{o_1,\dots,o_n\}$. By assumption, $S \subset O$ are the positions where $\la$ takes odd values. 
For convenience, let $\delta_i = 2n-1-i$ for $1 \leq i \leq 2n-1$ so that $\beta_i = \la_i + \delta_i$. 
There are $n$ even and $n-1$ odd $\delta_i$'s. 

Since $\beta_i$ for $i \in O$ is odd, 
$\delta_i$'s for $i \in S$ are even and $\delta_i$'s for $i \in O \setminus S$ are odd. Which means that there are $j$ even and $n-j$ odd $\delta_i$'s for $i \in O$.
Consequently, we must have $n-j$ even and $j-1$ odd $\delta_i$'s for $i \in E$.
Combining all this information, we infer that the number of odd parts in $\la$ is $j$ (from $S$) and $j-1$ (from $E$, since $\beta_i$'s for $i \in E$ are even), giving a total of $2j-1$.
But we had assumed that $\la$ has $2k+1$ odd parts, and therefore $j = k+1$.

This completes the proof of both the cases of the Proposition.
\end{proof}


\vskip 10pt

 \begin{thebibliography}{MVW}


\bibitem[D]{D} F. Digne: {\em Shintani descent and L-functions on Deligne Lusztig varieties,}
  Proc. Sympos. Pure Math. 47, part 1, (1987), 61-68.


\bibitem[F-H]{F-H} W. Fulton and J. Harris: {\em Representation theory, a first course},
  Graduate Texts in Mathematics, 129. Readings in Mathematics. Springer-Verlag, New York, 1991.
  
\bibitem[GAP]{GAP}
  The GAP~Group, \emph{GAP -- Groups, Algorithms, and Programming, 
  Version 4.10.2}; 
  2019,
  \url{https://www.gap-system.org}.

\bibitem[Ka]{Ka}
N. Kawanaka,  {\em Liftings of irreducible characters of finite classical groups. I.}
J. Fac. Sci. Univ. Tokyo Sect. IA Math. 28 (1981), no. 3, 851-861 (1982).

\bibitem[KLP]{KLP} S. Kumar, G. Lusztig and D. Prasad, {\em Characters of simplylaced nonconnected groups versus characters of nonsimplylaced connected groups.}
  Contemporary Math., vol 478, AMS, pp. 99-101.


  \bibitem[Li1]{Li1}D. E. Littlewood, {\em 
    The Theory of Group Characters and Matrix Representations of Groups.}
    Oxford University Press, New York 1940. 

  \bibitem[Li2]{Li2}D. E. Littlewood, {\em Modular representations of symmetric groups.}
  Proc. Roy. Soc. London Ser. A 209 (1951), 333-353. 

\bibitem[Lu]{Lu} G. Lusztig. {\em Left cells in Weyl groups.} Lie group representations, I (College Park, Md., 1982/1983),
  99-111, Lecture Notes in Math., 1024, Springer, Berlin, 1983.
  
\bibitem[P]{P} D. Prasad, {\em A character relationship on $\GL_n(\C)$}; Israel Journal, vol. 211, (2016), 257-270.

  \bibitem[Ro]{Ro} G. de B. Robinson,
    {\em Representation theory of the symmetric group.}
    Mathematical Expositions, No. 12. University of Toronto Press, Toronto 1961 vii+204 pp. 

  \bibitem[Sh]{Sh}
  T. Shintani: {\em Two remarks on irreducible characters of finite general linear groups.} J. Math. Soc. Japan 28 (1976), no. 2, 396-414. 

\end{thebibliography}
\end{document}